\documentclass[12pt]{amsart}
\usepackage{amssymb,amsmath, amsthm,latexsym,verbatim}
\usepackage{a4wide}
\usepackage{mathrsfs}
\usepackage[hypertex]{hyperref}
\usepackage{enumitem}
\usepackage{todonotes}

\theoremstyle{plain}

\newtheorem{lem}{Lemma}[section]
\newtheorem{theo}[lem]{Theorem}
\newtheorem{prop}[lem]{Proposition}
\newtheorem{cor}[lem]{Corollary}

\parskip=\medskipamount
\font\k=cmr7

  \newcommand {\ran}{\mbox{\k ran}}
  \newcommand {\comb}{\mbox{\k comb}}

  \newcommand {\pr}{\mbox{\k prime}}

  \newcommand {\ev}{\mbox{\k ev}}

  \newcommand {\C}{{\mathbb C}}
  \newcommand {\bH}{{\mathbb H}}
  \newcommand {\N}{{\mathbb N}}
  \newcommand {\R}{{\mathbb R}}
  \newcommand {\Z}{{\mathbb Z}}

  \newcommand {\af}{{\mathfrak a}}
  \newcommand {\bfr}{{\mathfrak b}}
  \newcommand {\gf}{{\mathfrak g}}
  \newcommand {\mf}{{\mathfrak m}}
  \newcommand {\kf}{{\mathfrak k}}

  \newcommand {\hf}{{\mathfrak h}}
  
  \newcommand {\nf}{{\mathfrak n}}

   \newcommand {\pf}{{\mathfrak p}}

  \newcommand {\Co}{{\mathcal C}}

 \newcommand {\cM}{{\mathcal M}}

\newcommand {\bs}{\backslash}
\newcommand {\wt}{\widetilde}

\renewcommand{\Im}{\operatorname{Im}}
\renewcommand{\Re}{\operatorname{Re}}
\newcommand{\Tr}{\operatorname{Tr}}
\newcommand{\Spec}{\operatorname{Spec}}
\newcommand{\End}{\operatorname{End}}

\newcommand{\tr}{\operatorname{tr}}

\newcommand{\Id}{\operatorname{Id}}

\newcommand{\I}{\operatorname{I}}

\newcommand{\vol}{\operatorname{vol}}

\newcommand{\SL}{\operatorname{SL}}
\newcommand{\GL}{\operatorname{GL}}
\newcommand{\SO}{\operatorname{SO}}
\newcommand{\SU}{\operatorname{SU}}

\newcommand{\Ad}{\operatorname{Ad}}

\renewcommand{\det}{\operatorname{det}}

\newcommand{\Rep}{\operatorname{Rep}}
\newcommand{\gr}{\mathrm{gr}}

\newcommand{\spectr}{\operatorname{spec}}

\begin{document}

\title[]
{On Fried's conjecture for compact hyperbolic manifolds}
\date{\today}

\author{Werner M\"uller}
\address{Universit\"at Bonn\\
Mathematisches Institut\\
Endenicher Allee 60\\
D -- 53115 Bonn, Germany}
\email{mueller@math.uni-bonn.de}

\keywords{Ruelle zeta function, analytic torsion}
\subjclass[2010]{Primary: 37C30, Secondary: 58J20}

\begin{abstract}
Fried's conjecture is concerned with the behavior of dynamical zeta functions
at the origin. For compact hyperbolic manifolds, Fried proved that  for
an orthogonal acyclic representation  of the fundamental group, 
the twisted Ruelle zeta function  is holomorphic at $s=0$ and its
value at $s=0$ equals the Ray-Singer analytic torsion. He also established
a more general result for orthogonal representations, which are not acyclic.
The purpose of the present paper is to extend Fried's result to arbitrary finite
dimensional representations of the fundamental group. The Ray-Singer analytic
torsion is replaced by the complex-valued torsion introduced by Cappell and
Miller.
\end{abstract}

\maketitle
\setcounter{tocdepth}{1}
\tableofcontents

\section{Introduction}
Let $X$ be a $d$-dimensional closed, oriented hyperbolic manifold. Then there
exists a discrete torsion free subgroup $\Gamma\subset\SO_0(d,1)$ such that
$X=\Gamma\bs\bH^d$, where $\bH^d=\SO_0(d,1)/\SO(d)$ is the $d$-dimensional 
hyperbolic space. Every $\gamma\in\Gamma\setminus\{e\}$ is hyperbolic and the
$\Gamma$-conjugacy class $[\gamma]$ corresponds to a unique closed geodesic
$\tau_\gamma$. Let $\ell(\gamma)$ denote the length of $\tau_\gamma$. A conjugacy
class is called prime if $\gamma$ is not a non-trivial power of some other
element of $\Gamma$. Let $\chi\colon\Gamma\to\GL(V_\chi)$ be a finite 
dimensional complex representation of $\Gamma$ and let $s\in\C$. 
Then the Ruelle zeta function 
$R(s,\chi)$ is defined by the following Euler product
\begin{equation}\label{ruelle-zeta1}
R(s,\chi):=\prod_{\substack{[\gamma]\not=e\\ [\gamma]\,\pr}}\det\left(\Id-\chi(\gamma)
e^{-s\ell(\gamma)}\right).
\end{equation}
The infinite product is absolutely convergent in a certain half plane $\Re(s)>C$
and admits a meromorphic extension to the entire complex plane \cite{Fr2},
\cite{Sp1}. The Ruelle zeta is a dynamical zeta function associated to the 
geodesic flow on the unit sphere bundle $S(X)$ of $X$. 
There are formal analogies with the zeta functions in number theory such as
the Artin $L$-function associated to a Galois representation. Analogues to
the role of zeta functions in number theory, one expects that special values
of the Ruelle zeta function provide a connection between the length spectrum of
closed geodesics and geometric and topological invariants of the manifold.

In \cite{Fr1}, Fried has established such a connection. To explain his result we
need to introduce some notation. Recall that a representation $\chi$ 
is called {\it acyclic}, if the cohomology $H^\ast(X,F_\chi)$ of $X$ with 
coefficients in the flat bundle $F_\chi\to X$ associated to $\chi$ vanishes. 
Let $\chi$ be an orthogonal acyclic representation. Then $F_\chi$ is equipped
with a canonical fibre metric which is compatible with the flat connection.
Let $\Delta_{k,\chi}$ be the Laplacian acting in the space $\Lambda^k(X,F_\chi)$
of $F_\chi$-valued $k$-forms. Regarded as operator in the space of $L^2$-forms,
it is essentially self-adjoint with a discrete spectrum $\Spec(\Delta_{k,\chi})$
consisiting of eigenvalues $\lambda$ of finite multiplicity $m(\lambda)$. 
Let $\zeta_k(s;\chi)=
\sum_{\lambda\in\Spec(\Delta_{k,\chi})}m(\lambda)\lambda^{-s}$ be the spectral zeta
function of $\Delta_{k,\chi}$ \cite{Shb}.  The series converges absolutely in the
half plane $\Re(s)>d/2$ and admits a meromorphic extension to the complex
plane, which is holomorphic at $s=0$. Then the Ray-Singer analytic torsion
 $T^{RS}(X,\chi)\in\R^+$ is defined by
\begin{equation}\label{analyt-tor}
\log T^{RS}(X,\chi):=\frac{1}{2}\sum_{k=1}^d (-1)^kk\frac{d}{ds}\zeta_k(s;\chi)
\big|_{s=0},
\end{equation}
\cite{RS}. Now we can state the result of  Fried
\cite[Theorem 1]{Fr1}. He proved that for an acyclic unitary representation
$\chi$ the Ruelle zeta function $R(s,\chi)$ is holomorphic at $s=0$ and
\begin{equation}\label{value-origin1}
|R(0,\chi)^\varepsilon|=T^{RS}(X,\chi)^2,
\end{equation}
where $\varepsilon=(-1)^{d-1}$ and the absolute value can be removed if $d>2$.
If $\chi$ is not acyclic, but still orthogonal, $R(s,\chi)$ may have a pole or 
zero at $s=0$. 
Fried \cite{Fr1} has determined the order of $R(s,\chi)$ at $s=0$ and the
leading coefficient of the Laurant expansion around $s=0$. Let 
$b_k(\chi):=\dim H^k(X,E_\chi)$. Assume that $d=2n+1$.  Put 
\[
h=2\sum_{k=0}^n(n+1-k)(-1)^kb_k(\chi).
\]
Then by \cite[Theorem 3]{Fr1}, the order of $R(s,\chi)$ at $s=0$ is $h$ and 
the leading term of the Laurent expansion of $R(s,\chi)$ at $s=0$ is 
\begin{equation}\label{laurent-exp1}
C(\chi)\cdot T^{RS}(X,\chi)^2s^h,
\end{equation}
where $C(\chi)$ is a constant that depends on the Betti numbers $b_k(\chi)$.  
In \cite[p. 66]{Fr4} Fried conjectured that \eqref{value-origin1} holds
for all compact locally symmetric manifolds $X$ and acyclic orthogonal 
bundles over $S(X)$. This conjecure was recently proved by Shu Shen \cite{Shu}.

Let $\chi$ be a unitary acyclic representation of $\Gamma$. Let 
$\tau(X,\chi)$ be the Reidemeister torsion \cite{RS}, \cite{Mu3}. It is 
defined in terms of a smooth triangulation of $X$. However, it is independent
of the particular $C^\infty$-triangulation. Since $\chi$
is acyclic, $\tau(X,\chi)$ is a topological invariant, i.e., it does not
depend on the metrics on $X$ and in $F_\rho$. 
By \cite{Ch}, \cite{Mu2} we have $T^{RS}(X,\chi)=\tau(X,\chi)$. Assume that 
$d$ is odd. Then \eqref{value-origin1} can be restated as
\begin{equation}\label{reidem-tor}
R(0,\chi)=\tau(X,\chi)^2.
\end{equation}
This provides an interesting relation between the length spectrum of closed
geodesics and a secondary topological invariant. 

Another class of interesting representations arises in the following way.
Let $G:=\SO_0(d,1)$. Let $\rho$ be a finite dimensional complex or real 
representation of $G$. Then $\rho|_\Gamma$ is a finite dimensional 
representation of $\Gamma$. In general, $\rho|_\Gamma$ is not an orthogonal
representation. However, the flat vector bundle $F_\rho$ associated with 
$\rho|_\Gamma$ can be equipped with a canonical fibre metric which allows
the use of methods of harmonic analysis to study the Laplace operators
$\Delta_{k,\rho}$. Put
\[
R(s,\rho):=R(s,\rho|_\Gamma).
\]
The behavior of $R(s,\rho)$ at $s=0$  has been studied by Wotzke \cite{Wo}. 
Let $\theta\colon G\to G$ be the Cartan involution of $G$ with respect to
$K=\SO(d)$. Let $\rho_\theta:=\rho\circ\theta$. Also denote by $T^{RS}(X,\rho)$
the analytic torsion of $X$ with respect to $\rho|_\Gamma$ and an admissible
metric in $F_\rho$.  Assume that $\rho\not\cong\rho_\theta$. Then Wotzke \cite{Wo} has proved that $R(s,\rho)$ is holomorphic at $s=0$ and 
\begin{equation}\label{value-origin2}
|R(0,\rho)|=T^{RS}(X,\rho)^2.
\end{equation}
If $\rho\cong\rho_\theta$, then $R(s,\rho)$ may have a zero or a pole at $s=0$.
Wotzke \cite{Wo} has also determined the order of $R(s,\rho)$ at $s=0$
and the coefficient of the leading term of the Laurent expansion of $R(s,\rho)$
at $s=0$. As in \eqref{laurent-exp1} the main contribution to the coefficient 
is the analytic torsion.

Let $\tau(X,\rho)$ be the Reidemeister torsion \cite{Mu3} of $X$ with respect to
$\rho|_\Gamma$.
If $\rho\not\cong\rho_\theta$, the cohomology $H^\ast(X,F_\rho)$ vanishes
\cite[Chapt. VII, Theorem 6.7]{BW}. Then $\tau(X,\rho)$ is independent of
the metrics on $X$ and in $F_\rho$. 
By \cite[Theorem 1]{Mu3} we have $T^{RS}(X,\rho)=\tau(X,\rho)$. Thus
\eqref{value-origin2} can be restated as
\begin{equation}\label{value-origin4}
|R(0,\rho)|=\tau(X,\rho)^2.
\end{equation}
This equality has interestig consequences for arithmetic subgroups $\Gamma$. 
Assume that there exists a $\Gamma$-invariant lattice $M_\rho\subset V_\rho$.
Let $\cM_\rho\to X$ be the associated local system of free $\Z$-modules of
finite rank. The cohomology $H^\ast(X,\cM_\rho)$ is a finitely generated 
abelian group. If $\rho_\Gamma$ is acyclic, $H^\ast(X,\cM_\rho)$ is a finite
abelian group. Denote by $|H^k(X,\cM_\rho)|$ the order of $H^k(X,\cM_\rho)|$.
By \cite[(1.4)]{Ch}, \cite[Sect. 2.2]{BV}, $\tau(X,\rho)$ 
can be expressed in terms of $|H^k(X,\cM_\rho)|$, $k=0,...,d$. Combined with 
\eqref{value-origin4} we get
\begin{equation}\label{value-origin5}
|R(0,\rho)|= \prod_{k=0}^d |H^k(X,\cM_\rho)|^{(-1)^{k+1}}.
\end{equation}
This is another interesting realtion between the length spectrum of $X$ and
topological invariants of $X$.

For arithmetic subgroups $\Gamma\subset G$, representations of $G$ with 
$\Gamma$-invariant lattices in the corresponding representation space exist. 
See \cite{BV}, \cite{MaM}.

The main purpose of this paper is to extend the above results about the 
behaviour of the Ruelle zeta function at $s=0$ to every finite dimensional
representation $\chi$ of $\Gamma$. To this end we use a complex
version $T^\C(X,\chi)$ of the analytic torsion, which was introduced by 
Cappell and Miller \cite{CM}. It is defined in terms of the flat Laplacians
$\Delta_{k,\chi}^\sharp$, $k=0,...,d$, which are obtained by coupling the 
Laplacian $\Delta_k$ on $k$-forms to the flat bundle $F_\chi$ (see section
\ref{sec-coupling} for its definition). In general, the flat Laplacian 
$\Delta_{k,\chi}^\sharp$ is not self-adjoint. However, its principal symbol 
equals the principal symbol of a Laplace type operator. Therefore, it has 
good spectral properties which allows to carry over most of the results from
the self-adjoint case. The Cappell-Miller torsion $T^\C(X,\chi)$ is defined
as an element of the determinant line
\[
T^\C(X,\chi)\in\det H^\ast(X,F_\chi)\otimes(\det H_\ast(X,F_\chi))^\ast.
\]
For an acyclic representation $T^\C(X,\chi)$  is a complex number and
\begin{equation}\label{abs-value}
|T^\C(X,\chi)|=T^{RS}(X,\chi)^2,
\end{equation}
 where $T^\C(X,\chi)$ is the Ray-Singer analytic
torsion with respect to any choice of a fibre metric in $F_\chi$. Since
$\chi$ is acyclic, $T^{RS}(X,\chi)$ is independent of the choice of the metric 
in $F_\chi$. 

Let $V_0^k$ be the generalized eigenspace of $\Delta_{k,\chi}^\sharp$, 
$k=0,...,d$, with generalized eigenvalue $0$. Let $d^{\ast,\sharp}_\chi$ be the 
coupling of the codifferential $d^\ast_\chi\colon\Lambda^\ast(X)\to\Lambda^\ast(X)$
to the flat bundle $F_\chi$. Then $(V_0^\ast,d_\chi,d^{\ast,\sharp}_\chi)$ is a 
double complex in the sense of \cite[\S 6]{CM}. Let
\begin{equation}
T_0(X,\chi)\in\det H^\ast(X,F_\chi)\otimes(\det H_\ast(X,F_\chi))^\ast.
\end{equation}
be its torsion \cite[\S 6]{CM}. 
We note that $T^\C(X,\chi)$ and $T_0(X,\chi)$ are both non-zero elements of the
determinant line $\det H^\ast(X,F_\chi)\otimes(\det H_\ast(X,F_\chi))^\ast$. 
Hence there exists $\lambda\in\C$ with $T^\C(X,\chi)=\lambda T_0(X,\chi)$.
Set
\[
\frac{T^\C(X,\chi)}{T_0(X,\chi)}:=\lambda.
\]
Put
\begin{equation}
h_k:=\dim V_0^k,\quad k=0,...,d.
\end{equation}
Furthermore, let $d=2n+1$ and put
\begin{equation}\label{order-sing1}
h:=\sum_{k=0}^n(d+1-2k)(-1)^kh_k
\end{equation}
and
\begin{equation}\label{const1}
C(d,\chi):=\prod_{k=0}^{d-1}\prod_{p=k}^{d-1} \left(2(n-p)\right)^{(-1)^kh_k}.
\end{equation}
Then our main result is the following theorem.
\begin{theo}\label{theo-sing}
Let $\chi$ be a finite dimensional complex representation of $\Gamma$. Let
$h$ be defined by \eqref{order-sing1}. 
Then the order of the singularity of $R(s,\chi)$ at $s=0$ is $h$ and
\begin{equation}\label{origin}
\lim_{s\to0} s^{-h}R(s,\chi)=C(d,\chi)\cdot\frac{T^\C(X,\chi)}{T_0(X,\chi)}.
\end{equation}
\end{theo} 
Now choose a triangulation of $X$. Let $\tau_{\comb}(X,\chi)\in
\det H^\ast(X,F_\chi)
\otimes(\det H_\ast(X,F_\chi))^\ast$ be the combinatorial torsion defined Cappell 
and Miller \cite[Sect. 9]{CM}. It is independent of the choice of the
triangulation. By \cite[Theorem 10.1]{CM} we have
$T^\C(X,\chi)=\tau_{\comb}(X,\chi)$. Thus we can restate Theorem \ref{theo-sing}
as
\begin{equation}\label{origin1}
\lim_{s\to0} s^{-h}R(s,\chi)=C(d,\chi)\cdot\frac{\tau_{\comb}(X,\chi)}{T_0(X,\chi)}.
\end{equation}
If $\chi$ is acyclic, then $T^\C(X,\chi)$, $T_0(X,\chi)$ and 
$\tau_{\comb}(X,\chi)$ 
are complex numbers and on the right hand side of \eqref{origin} and 
\eqref{origin1} appear quotients of complex numbers. 
 
Now we apply Theorem \ref{theo-sing} to representations of $\Gamma$ which
are restrictions of representations of $G$. Then we get
\begin{cor}\label{cor-origin1}
Let $\rho\in\Rep(G)$ and assume that $\rho\not\cong\rho_\theta$. Then $R(s,\rho)$
is holomorphic at $s=0$ and
\begin{equation}\label{value-origin6}
R(0,\rho)=C(d,\rho)\cdot\frac{T^\C(X,\rho)}{T_0(X,\rho)}.
\end{equation}
\end{cor}
Using \eqref{value-origin2} and \eqref{abs-value}, it follows that
\begin{equation}
|T_0(X,\rho)|=C(d,\rho).
\end{equation}
Let $d=3$. Then $\bH^3\cong\SL(2,\C)/\SU(2)$. For $m\in\N$ let $\rho_m\colon
\SL(2,\C)\to \SL(S^m(\C^2))$ be the $m$-th symmetric power of the standard 
representation of $\SL(2,\C)$ on $\C^2$. For a compact, oriented 
hyperbolic $3$-manifold $X=\Gamma\bs\bH^3$ and the
representations $\rho_m$, $m\in\N$, Corollary \ref{cor-origin1} was proved 
by J. Park \cite[(5.5)]{Pa}. He also determined the constant $C(3,\rho_m)$ and
$|T_0(X,\rho_m)|$. By \cite[Prop. 5.1]{Pa} we have
\begin{equation}\label{value}
\begin{split}
&h_0=1\quad\text{and}\quad|T_0(X,\rho_m)|=2,\quad\text{if}\;m\;\text{is}\;
\text{even}\\
&h_0=0\quad\text{and}\;\;\quad T_0(X,\rho_m)=1,\quad\text{if}\;m\;\text{is}\;
\text{odd}.
\end{split}
\end{equation}
Moreover $C(3,\rho_m)=(-4)^{h_0}$. The order $h$ of $R(s,\rho_m)$ at
$s=0$ is zero. Thus by \eqref{order-sing1} we have $h_1=2h_0$. Note that
$\rho_m$ is acyclic. Let $\Delta_{k,\rho_m}$ be the usual Laplacian in
$\Lambda^k(X,F_{\rho_m})$ with respect to the admissible metric in $F_{\rho_m}$.
Then for $m\in\N$ even we have
\begin{equation}
\ker\Delta_{k,\rho_m}=0,\quad \ker\Delta_{k,\rho_m}^\sharp\neq 0,\quad k=0,...,
\end{equation}
which shows that for acyclic representations $\chi$, in general, the flat 
Laplacian $\Delta_{k,\chi}^\sharp$  need not be invertible. 

Again  we can replace $T^\C(X,\rho)$ in \eqref{value-origin6} by the 
combinatorial torsion $\tau_{\comb}(X,\rho)$. However, in the present case we
can replace the combinatorial torsion by the complex Reidemeister torsion.
Since $G$ is a connected semisimple Lie group and $\rho$ a representation of 
$G$, it follows from \cite[Lemma 4.3]{Mu3} that $\rho$ is actually a 
representation in $\SL(n,\C)$. This implies that the complex Reidemeister
torsion $\tau^\C(X,\rho)\in\C^\ast/\{\pm1\}$ can be 
defined as the usual Reidemeister torsion $\tau(X,\rho)$ 
\cite[Definition 1.1]{RS}, where the absolute value of the 
determinant is deleted. In particular, we have
\begin{equation}
|\tau^\C(X,\rho)|=\tau(X,\rho).
\end{equation}
Using \eqref{value-origin6} we get
\begin{equation}\label{compl-reidem}
R(0,\rho)=C(d,\rho)\cdot\frac{\tau^\C(X,\rho)^2}{T_0(X,\rho)}.
\end{equation}

Another case to which Theorem \ref{theo-sing} can be applied are deformations
of unitary acyclic representations. Let $\Rep(\Gamma,\C^n)$ be the set of all
$n$-dimensional complex representations of $\Gamma$ equipped with usual
topology.  Let $\Rep^u_0(\Gamma,\C^n)
\subset \Rep(\Gamma,\C^n)$ be the subset of all unitary acyclic representations
(see section 6.2). By \cite[Theorem 1.1]{FN}, we have $\Rep^u_0(\Gamma,\C^n)
\neq\emptyset$. There exists a neighborhood $V$ of $\Rep^u_0(\Gamma,\C^n)$ in
$\Rep(\Gamma,\C^n)$ such that $\Delta_{\chi}^\sharp$ is invertible for all
$\chi\in V$. Using Theorem \ref{theo-sing}, we get
\begin{prop}\label{prop-acyclic}
Let $\chi\in V$. Then $R(s,\chi)$ is regular at $s=0$ and 
\[
R(0,\chi)=T^\C(X,\chi).
\]
\end{prop}
This proposition was first proved by P. Spilioti \cite{Sp3} using the odd
signature operator \cite{BK1}. She also discussed the relation with the 
refined analytic torsion.

\section{Coupling differential operators to a flat bundle}\label{sec-coupling}
\setcounter{equation}{0}
We recall a construction of the flat extension of a differential operator 
introduced in \cite{CM}. 
Let $X$ be a smooth manifold and $E_1$ and $E_2$ complex vector bundles over 
$X$. Let 
\[
D\colon C^\infty(X,E_1)\to C^\infty(X,E_2)
\]
be a differential operator. Let $F\to X$ be a flat vector bundle. Then there
is a canonically operator 
\[
D_F^\sharp\colon C^\infty(X,E_1\otimes F)\to C^\infty(X,E_2\otimes F)
\]
associated to $D$, which is defined as follows. Let $U\subset X$ be an open
subset such that $F|_U$ is trivial. Let $s_1,\dots,s_k\in C^\infty(U,F|_U)$ be
a local frame field of flat sections. Every section $\varphi$ of 
$(E_1\otimes F)|_U$ can be written as 
\[
\varphi=\sum_{i=1}^k \psi_i\otimes s_i
\]
for some sections $\psi_1,...,\psi_k\in C^\infty(U,E_1|U)$. Then define 
\[
D_F^\sharp|_U\colon C^\infty(U,(E_1\otimes F)|_U)\to C^\infty(U,(E_2\otimes F)|_U)
\]
by
\[
(D_F^\sharp|_U)(\varphi):=\sum_{i=1}^k D(\psi_i)\otimes s_i.
\]
Let $s_1^\prime,...,s_k^\prime$ be another local frame field of flat sections of
$F|_U$. Then $s_i=\sum_{j=1}^k f_{ij}s_j^\prime$, $i=1,...,k$, with $f_{ij}\in 
C^\infty(U)$, and it follows that
the transition functions $f_{ij}$ are constant. Since $D$ is linear, 
$(D_F^\sharp|_U)(\varphi)$ is independent of the choice of the local frame 
field of flat sections and therefore, $D_F^\sharp$ is globally well defined. 
Let $\sigma(D)$
be the principal symbol of $D$. Then the principal symbol $\sigma(D_F^\sharp)$
of $D_F^\sharp$ is given by $\sigma(D_F^\sharp)=\sigma(D)\otimes \Id_F$. Thus if
$D$ is elliptic, then $D_F^\sharp$ is also an elliptic differential operator. 

As an example consider a Riemannian manifold $X$ and the Laplace 
operator $\Delta_p$ on $p$-forms. Let $F$ be a flat bundle over $X$. Denote 
by $\Lambda^p(X,F)$ the space of smooth $F$-valued $p$-forms, i.e., 
$\Lambda^p(X,F)=C^\infty(X,\Lambda^pT^\ast(X)\otimes F)$. By the construction 
above we obtain the flat Laplacian 
$\Delta_{p,F}^\sharp\colon \Lambda^p(X,F)\to\Lambda^p(X,F)$. If the flat bundle
is fixed, we will denote the flat Laplacian simply by $\Delta_p^\sharp$. 
The flat Laplacian can be also described as the usual Laplacian. Let
$d_F\colon \Lambda^{p-1}(X,F)\to \Lambda^p(X,F)$ be the exterior derivative
defined as above. Let $\star\colon \Lambda^p(X)\to \Lambda^{n-p}(X)$ denote the
Hodge $\star$-operator. Then the flat extension 
\[
d_F^{\ast,\sharp}\colon \Lambda^p(X,F)\to \Lambda^{p-1}(X,F)
\]
of the co-differential $d^\ast$ is given by
\[
d_F^{\ast,\sharp}=(-1)^{np+n+1}(\star\otimes \Id_F)\circ d_F\circ(\star\otimes \Id_F).
\]
Then $d_F^{\ast,\sharp}$ satisfies $d_F^{\ast,\sharp}\circ d_F^{\ast,\sharp}=0$ and we
have
\[
\Delta_{F}^\sharp=(d_F+d_F^{\ast,\sharp})^2.
\]
If we choose a Hermitian fibre metric on $F$, we can define the usual Laplace
operator $\Delta_F$ in $\Lambda^p(X,F)$, which is defined by
\[
\Delta_F=(d_F+d_F^\ast)^2=d_Fd_F^\ast+d_F^\ast d_F,
\]
which is formally self-adjoint. Now note that $d_F^{\ast,\sharp}=d_F^\ast+B$,
where $B$ is a smooth homomorphism of vector bundles. Thus it follows that 
$\Delta_F^\sharp=\Delta_F+(Bd_F+d_FB)$. Thus the principal symbol 
$\sigma(\Delta_F^\sharp)(x,\xi)$ of $\Delta_F^\sharp$ is given by 
\begin{equation}\label{symbol}
\sigma(\Delta_F^\sharp)(x,\xi)=\|\xi\|^2_x\Id_{\Lambda^p T^\ast_x(X)\otimes F_x},
\quad x\in X,\;\xi\in F_x.
\end{equation}

More generally, let $E\to X$ be a Hermitian vector bundle over $X$. Let 
$\nabla$ be a covariant derivative in $E$ which is compatible 
with the Hermitian metric. We denote by $C^\infty(X,E)$ the space of smooth 
sections of $E$. Let 
\[
\Delta_E=\nabla^*\nabla
\]
be the Bochner-Laplace operator associated to the connection $\nabla$ and the
Hermitian fiber metric. Then $\Delta_E$ is a second order elliptic differential
operator. Its leading symbol
$\sigma(\Delta_E)\colon \pi^*E\to\pi^*E$, where $\pi$ is the projection of
$T^*X$,  is given by
\begin{equation}\label{symbol0}
\sigma(\Delta_E)(x,\xi)=\parallel\xi\parallel^2_x\cdot\Id_{E_x},\quad x\in X,\;
\xi\in T_x^*X.
\end{equation}
Let $F\to X$ be a flat vector bundle and 
\[
\Delta_{E\otimes F}^\sharp\colon C^\infty(X,E\otimes F)\to C^\infty(X,E\otimes F)
\]
the coupling of $\Delta_E$ to $F$. Then the principal symbol of 
$\Delta_{E\otimes F}^\sharp$ is given by
\begin{equation}\label{prin-symp}
\sigma(\Delta_{E\otimes F}^\sharp)(x,\xi)=\|\xi\|^2_x\cdot\Id_{E_x\otimes F_x}.
\end{equation}

\section{Regularized determinants and analytic torsion}\label{sec-det}

Let $\Delta_E$ be as above. Let 
\[
P\colon C^\infty(X,E)\to C^\infty(X,E)
\]
be an elliptic second order differential operator which is a perturbation of 
$\Delta_E$ by a first order differential operator, i.e., 
\begin{equation}\label{perturb}
P=\Delta_E+D,
\end{equation}
where $D\colon C^\infty(X,E)\to C^\infty(X,E)$ is a first oder differential 
operator. This implies that $P$ is an elliptic second order differential 
operator with leading symbol $\sigma(P)(x,\xi)$ given by
 \begin{equation}\label{lead-symb}
\sigma(P)(x,\xi):=\|\xi\|^2_x\cdot\Id_{E_x}.
\end{equation}
Though $P$ is not self-adjoint in general, it still has nice spectral properties
\cite[Chapt. I, \S 8]{Shb}. We recall the basic facts. For $I\subset[0,2\pi]$ 
let
\begin{equation}\label{cone}
\Lambda_I=\{re^{i\theta}\colon 0\le r<\infty,\;\theta\in I\}.
\end{equation}
The following lemma describes the 
structure of the spectrum of $P$. 
\begin{lem}\label{spec1}
For every $ 0<\varepsilon<\pi/2$ there exists $R>0$ such that the spectrum of
$P$ is contained in the set $B_R(0)\cup \Lambda_{[-\varepsilon,\varepsilon]}$. 
Moreover the spectrum of $P$ is discrete.
\end{lem}
\begin{proof} The first statement follows from \cite[Theorem 9.3]{Shb}. 
The discreteness of the spectrum follows from \cite[Theorem 8.4]{Shb}. 
\end{proof}
For $\lambda\in\C\setminus\spectr(P)$ let $R_\lambda(P):=(P-\lambda\Id)^{-1}$
be the resolvent. Given $\lambda_0\in\spectr(P)$, let
$\Gamma_{\lambda_0}$ be a small circle around $\lambda_0$ which
contains no other points of $\spectr(P)$. Put
\begin{equation}\label{project}
\Pi_{\lambda_0}=\frac{i}{2\pi}\int_{\Gamma_{\lambda_0}}R_\lambda(P)\;d\lambda.
\end{equation}
Then $\Pi_{\lambda_0}$ is the projection onto the {\it root subspace} 
$V_{\lambda_0}$.
This is a finite-dimensional subspace of $C^\infty(X,E)$ which
is invariant under $P$ and there exists $N\in\N$ such that 
$(P-\lambda_0\I)^NV_{\lambda_0}=0$. Furthermore, there is a closed complementary
subspace $V_{\lambda_0}^\prime$ to $V_{\lambda_0}$ in $L^2(X,E)$ 
which is invariant under the 
closure $\bar P$ of $P$ in $L^2$ and the restriction of $(\bar P-\lambda_0\I)$ 
to $V_{\lambda_0}^\prime$ has a bounded inverse. The {\it algebraic multiplicity}
$m(\lambda_0)$ of $\lambda_0$ is defined as
\[
m(\lambda_0):=\dim V_{\lambda_0}.
\]
Moreover
$L^2(X,E)$ is the closure of the algebraic direct sum of finite-dimensional 
$P$-invariant subspaces $V_k$
\begin{equation}\label{directsum}
L^2(X,E)=\overline{\bigoplus_{k\ge 1} V_k}
\end{equation}
such that the restriction of $P$ to $V_k$ has a unique eigenvalue $\lambda_k$,
for each $k$ there exists $N_k\in\N$ such that $(P-\lambda_k\I)^{N_k}V_k=0$, 
and $|\lambda_k|\to\infty$.  
In general, the sum (\ref{directsum}) is not a sum
of mutually orthogonal subspaces. See \cite[Sect. 2]{Mu1} for details.

Recall that an angle $\theta\in[0,2\pi)$ is called an {\it Agmon angle} for 
$P$, if there exists $\varepsilon>0$ such that
\begin{equation}\label{agmon-angle}
\spectr(P)\cap \Lambda_{[\theta-\varepsilon,\theta+\varepsilon]}=\emptyset.
\end{equation}
By Lemma \ref{spec1} it is clear that an Agmon angle always exists for $P$.
Assume that $P$ is invertible. Choose an Agmon angle for $P$. Define the
complex power $P_\theta^{-s}$, $s\in\C$, as in \cite[\S10]{Shb}.
For $\Re(s)>n/2$, the complex power $P_\theta^{-s}$ is a
trace class operator and the zeta function $\zeta_\theta(s,P)$ of $P$ is 
defined by
\begin{equation}\label{zeta-fct}
\zeta_\theta(s,P):=\Tr(P_\theta^{-s}),\quad \Re(s)>\frac{n}{2}.
\end{equation}
The zeta function admits a meromorphic extension to the entire complex plane 
which is holomorphic at $s=0$ \cite[Theorem 13.1]{Shb}. Let $R_\theta:=
\{\rho e^{i\theta}\colon \rho\in\R^+\}$. 
Denote by $\log_\theta(\lambda)$ the branch of the logarithm in 
$\C\setminus R_\theta$ with $\theta<\Im\log_\theta <\theta+2\pi$. We enumerate the
eigenvalues of $P$ such that 
\[
\Re(\lambda_1)\le\Re(\lambda_2)\le \cdots\le\Re(\lambda_k)\le\cdots.
\] 
By Lidskii's
theorem \cite[Theorem 8.4]{GK} if follows that for $\Re(s)>n/2$ we have
\begin{equation}\label{lidskii}
\zeta_\theta(s,P)=\Tr(P_\theta^{-s})=\sum_{k=1}^\infty m(\lambda_k) (\lambda_k)_\theta^{-s},
\end{equation}
where $(\lambda_k)_\theta^{-s}=e^{-s\log_\theta(\lambda_k)}$. We will need a different
description of the zeta functionThe zeta function in terms of the heat 
operator $e^{-tP}$, which can be defined
using the functional calculus developed in \cite[Sect. 2]{Mu1} by 
\begin{equation}\label{heat-op}
e^{-tP}:=\frac{i}{2\pi}\int_\Gamma e^{-t\lambda^2}(P^{1/2}-\lambda)^{-1}d\lambda,
\end{equation}
where $\Gamma\subset\C$ is the same contour as in \cite[(2.18)]{Mu1}. As in
\cite[Lemma 2.4]{Mu1} one can show that $e^{-tP}$ is an integral operator with
a smooth kernel. By \cite[Prop. 2.5]{Mu1} it follows that $e^{-tP}$ is a trace
class operator. Using Lidskii's theorem as above we get
\begin{equation}\label{heat-trace}
\Tr(e^{-tP})=\sum_{k=1}^\infty m(\lambda_k) e^{-t\lambda_k}.
\end{equation}
The absolute convergence of the right hand side follows from Weyl's law
\cite[Lemma 2.2]{Mu1}. Assume that there exists $\delta>0$ such that
$\Re(\lambda_k)\ge\delta$ for all $k\in\N$. Then by \eqref{heat-trace} and
Weyl's law it follows that there exist $C,c>0$ such that
\begin{equation}\label{exp-decay}
|\Tr(e^{-tP})|\le C e^{-ct}
\end{equation}
for $t\ge 1$. Since $\spectr(P)$ is contained in the half plane $\Re(s)>0$,
we can choose the Agmon angle as $\theta=\pi$. Using the asymptotic 
expansion of $\Tr(e^{-tP})$ as $t\to 0$, it follows from \eqref{lidskii}
and \eqref{exp-decay} that 
\[
\zeta(s,P)=\frac{1}{\Gamma(s)}\int_0^\infty \Tr(e^{-tP})t^{s-1}dt
\]
for $\Re(s)>n/2$.

Then the regularized
determinant of $P$ is defined by
\begin{equation}\label{determ}
\det_\theta(P):=\exp\left(-\frac{d}{ds}\zeta_\theta(s,P)\Big|_{s=0}\right).
\end{equation} 
As shown in \cite[3.10]{BK1}, $\det_\theta(P)$ is independent of $\theta$.
Therefore we will denote the regularized determinant simply by $\det(P)$. 

Assume that the vector bundle $E$ is $\Z/2\Z$-graded, i.e., $E=E^+\oplus E^-$ 
and $P$ preserves the grading, i.e., assume that with respect to the 
decomposition
\[
C^\infty(Y,E)=C^+(Y,E^+)\oplus C^\infty(Y,E^-)
\]
$P$ takes the form 
\[
P=\begin{pmatrix}P^+&0\\0&P^-\end{pmatrix}. 
\]
Then we define the graded determinant $\det_\gr(P)$ of $P$ by
\begin{equation}\label{grdet}
\det_{\gr}(P)=\frac{\det(P^+)}{\det(P^-)}.
\end{equation}

Next we introduce the analytic torsion defined in terms of the 
non-selfadjoint operators $\Delta_{p,\chi}^\sharp$. We use the definition given in 
\cite[section 8]{CM}.
Recall that the principal symbol of $\Delta_{p,\chi}^\sharp$ is given by 
\eqref{prin-symp}. Therefore, $\Delta_{p,\chi}^\sharp$ satisfies the assumptions
of section \ref{sec-det}.

 Let $r>0$ be such that $\Re(\lambda)\neq r$ for
all generalized eigenvalues $\lambda$ of $\Delta_{p,\chi}^\sharp$. Let 
$\Pi_{p,r}$ be the spectral projection on the span of the generalized 
eigenvectors with eigenvalues with real part less than $r$. Let
$\Delta_{p,\chi,r}^\sharp:=(1-\Pi_{p,r})\Delta_{p,\chi}^\sharp$. Let $S(p,\chi,r)$ be
the set of all nonzero generalized eigenvalues with real part less than $r$. 
Furthermore, let $V^p_0$ be the generalized eigenspace of $\Delta^\sharp_{p,\chi}$
with generalized eigenvalues $0$. Then $(V_0^\ast,d,d^{\ast,\sharp})$ is double
complex in the sense of \cite{CM}. Let
\begin{equation}\label{double}
T_0(X,\chi)\in(\det H^\ast(X,F_\chi))\otimes(\det H_\ast(X,F_\chi))^\ast
\end{equation} 
be the torsion of the double complex. Then the Cappell-Miller torsion  is
defined by
\begin{equation}\label{anal-tor}
T^\C(X,\chi):=\prod_{p=1}^d\det(\Delta_{p,\chi,r}^\sharp)^{(-1)^{p+1}p}\cdot
\prod_{p=1}^d\Bigl(\prod_{\lambda\in S(p,\chi,r)}\lambda^{m(\lambda)}\Bigr)^{(-1)^{p+1}p}\cdot 
T_0(X,\chi),
\end{equation}
where $m(\lambda)$ denotes the algebraic multiplicity of $\lambda$. 
Let $\Pi_{k,0}$ be the spectral projection on the generalized eigenspace of
$\Delta_{k,\chi}^\sharp$ with generalized eigenvalue $0$. Let 
\[
(\Delta_{k,\chi}^\sharp)^\prime:=(\Id-\Pi_{k,0})\Delta_{k,\chi}^\sharp.
\]
If we choose an Agmon angle we can also write
\begin{equation}\label{anal-tor1}
T^\C(X,\chi)=\prod_{k=1}^d\left[\det(\Delta_{k,\chi}^\sharp)^\prime\right]^{(-1)^{k+1}k}
\cdot T_0(X,\chi).
\end{equation}
If $\chi$ is acyclic, i.e., $H^\ast(X,E_\chi)=0$, then $T_0(X,\chi)$ and 
$T^\C(X,\chi)$ are complex numbers.

\section{Twisted Ruelle and Selberg zeta functions}
\setcounter{equation}{0}
In this section we consider compact oriented hyperbolic manifolds of odd 
dimension $d=2n+1$ and we recall some basic facts about Ruelle and Selberg 
type zeta functions. 

We need a  more general class of Ruelle zeta functions than the one defined
by \eqref{ruelle-zeta}.
To begin with we fix some notation. Let $G=\SO_0(d,1)$ and $K=\SO(d)$. 
Then $G/K$ equipped with the normalized invariant metric is isometric
to the $d$-dimensional hyperbolic space $\bH^d$. Let
$G=KAN$ be the standard Iwasawa decompositon. Let $M$ be the centralizer of
$A$ in $K$. Then $M\cong\SO(d-1)$. Denote by $\gf$, $\kf$, $\mf$, $\nf$, and 
$\af$
the Lie algebras of $G$, $K$, $M$, $N$, and $A$, respectively. Let $W(A)\cong
\Z/2\Z$ be the Weyl group of $(\gf,\af)$. 

Let $\Gamma\subset G$ be a discrete, torsion free, cocompact subgroup. Then
$\Gamma$ acts fixed point free on $\bH^d$. The quotient $X=\Gamma\bs\bH^n$ 
is a closed, oriented hyperbolic manifold and each such
manifold is of this form. Given $\gamma\in \Gamma$,
we denote by $[\gamma]$ the $\Gamma$-conjugacy class of $\gamma$. The set of 
all
conjugacy classes of $\Gamma$ will be denoted by $C(\Gamma)$. Let 
$\gamma\not=1$. Then there exist $g\in G$, $m_\gamma\in M$, and 
$a_\gamma\in A^+$ such that
\begin{equation}\label{conjug}
g\gamma g^{-1}=m_\gamma a_\gamma.
\end{equation}
By \cite[Lemma 6.6]{Wa}, $a_\gamma$ depends only on $\gamma$ and $m_\gamma$ 
is determined up to conjugacy in $M$.  By definition
there exists $\ell(\gamma)>0$  such that
\begin{equation}\label{length}
a_\gamma=\exp\left(\ell(\gamma)H\right).
\end{equation}
Then $\ell(\gamma)$ is the length of the unique closed geodesic in $X$ 
that corresponds to the conjugacy class $[\gamma]$. An element 
$\gamma\in\Gamma-\{e\}$ is called primitive, if it can not be written as 
$\gamma=\gamma_0^k$ for some $\gamma_0\in\Gamma$ and $k>1$. For every 
$\gamma \in\Gamma-\{e\}$ there exist a unique primitive element $\gamma_0\in
\Gamma$ and $n_\Gamma(\gamma)\in\N$ such that $\gamma=\gamma_0^{n_\Gamma(\gamma)}$.
We recall that for $R>0$ we have
\begin{equation}\label{conj-est}
\#\left\{[\gamma]\in C(\Gamma)\colon \ell(\gamma)\le R\right\}\ll e^{(n-1)R}
\end{equation}
\cite[(1.31)]{BO}. We also need the following auxiliary lemma.
\begin{lem}\label{repr-estim}
Let $\chi\colon\Gamma\to\GL(V)$ be a finite dimensional representation of
$\Gamma$. There exist $C,c>0$ such that
\begin{equation}\label{trace-est}
|\tr(\chi(\gamma))|\le C e^{c\ell(\gamma)},\quad \forall\gamma\in\Gamma-\{e\}.
\end{equation}
\end{lem}
For the proof see \cite[Lemma 3.3]{Sp1}. Let $\theta\colon\gf\to\gf$ be the
Cartan involution with respect to $\kf$. 
Let $\bar\nf=\theta\nf$ be the negative root space. Let $\chi\colon \Gamma\to
\GL(V)$ be a finite dimensional complex representation. For $\sigma\in\widehat
M$ and $s\in\C$ with $\Re(s)\gg 0$ the twisted Selberg zeta function is defined 
by
\begin{equation}\label{selb-zeta}
Z(s;\sigma,\chi):=\prod_{\substack{[\gamma]\not=e\\ [\gamma]\,\pr}}\prod_{k=0}^\infty
\det\left(1-\left(\chi(\gamma)\otimes\sigma(m_\gamma)\otimes S^k
\left(\Ad(m_\gamma a_\gamma)_{\overline\nf}\right)\right)
e^{-(s+\|\rho\|)\ell(\gamma)}\right),
\end{equation}
where $[\gamma]$ runs over the primitive $\Gamma$-conjugacy classes and
$S^k\left(\Ad(m_\gamma a_\gamma)_{\overline\nf}\right)$ denotes the $k$-th
symmetric power of the adjoint map $\Ad(m_\gamma a_\gamma)$ restricted to 
$\bar\nf$. It follows from \eqref{conj-est} and \eqref{trace-est} that there
exists $C>0$ such that the product converges absolutely and uniformly on campact
subsets of the half-plane $\Re(s)>C$. See \cite[Prop 3.4]{Sp1}. In the same way
the twisted Ruelle zeta function $R(s;\sigma,\chi)$ is defined by
\begin{equation}\label{ruelle-zeta}
R(s;\sigma,\chi):=\prod_{\substack{[\gamma]\not=e\\ [\gamma]\,\pr}}
\det\left(1-\left(\chi(\gamma)\otimes\sigma(m_\gamma)\right)
e^{-(s+|\rho|)\ell(\gamma)}\right).
\end{equation}
By \cite[Prop. 3.5]{Sp1} the product converges absolutely and uniformly
in some half-plane $\Re(s)>C$. Furthermore, $Z(s;\sigma,\chi)$ and
$R(s;\sigma,\chi)$ admit meromorphic extensions to the entire complex plane
\cite{Sp1} and satisfy functional equations \cite{Sp2}. For unitary 
representations $\chi$, these results were proved by Bunke and Olbrich 
\cite{BO}. The main technical tool is the Selberg trace formula. For the 
extension to the non-unitary case the Selberg trace formula is replaced
by a Selberg trace formula for non-unitary twists, developed in \cite{Mu1}.
The proofs are similar except that on has to deal with non-self-adjoint 
operators.

There are also expressions of the zeta functions in terms of determinants of
certain elliptic operators. To explain the formulas we need to recall the
definition of the relevant differential operators. Given $\tau\in\widehat K$,
let $\wt E_\tau\to\wt X$ be the homogeneous vector bundle associated to $\tau$
and let $E_\tau:=\Gamma\bs\wt E_\tau$ be the corresponding locally homogeneous
vector bundle over $X$. Denote by $C^\infty(X,E_\tau)$ the space of smooth
sections of $E_\tau$. There is a canonical isomorphism
\begin{equation}\label{iso-sect}
C^\infty(X,E_\tau)\cong (C^\infty(\Gamma\bs G)\otimes V_\tau)^K
\end{equation}
\cite[\S 1]{Mia}.
Let $\Omega\in{\mathcal Z}(\gf_\C)$ be the Casimir element and denote by 
$R_\Gamma$ the right regular representation of $G$ in $C^\infty(\Gamma\bs G)$.
Then $R_\Gamma(\Omega)$ acts on the right hand side of \eqref{iso-sect} and
via this isomorphism, defines an operator in $C^\infty(X,E_\tau)$. We denote
the operator induced by $-R_\Gamma(\Omega)$ by $A_\tau$.

Denote by $\wt\nabla^\tau$ the canonical connection in
$\wt E_\tau$ and let $\nabla^\tau$ be the induced connection in $E_\tau$. Let
$\Delta_\tau:=(\nabla^\tau)^\ast\nabla^\tau$ be the associated Bochner-Laplace 
operator acting in $C^\infty(X,E_\tau)$  .
Let $\Omega_K\in{\mathcal Z}(\kf_\C)$ be the Casimir element of $K$. Assume
that $\tau$ is irreducible. 
Let $\lambda_\tau:=\tau(\Omega_K)$ denote the Casimir eigenvalue of $\tau$.
Then we have 
\begin{equation}
A_\tau:=\Delta_\tau-\lambda_\tau\Id.
\end{equation}
\cite[\S 1]{Mia}.
Thus $A_\tau$ is a formally self-adjoint second order elliptic differential
operator. Let $F_\chi\to X$ be the flat vector bundle defined by $\chi$. Let 
\[
A^\sharp_{\tau,\chi}\colon C^\infty(X,E_\tau\otimes F_\chi)\to 
C^\infty(X,E_\tau\otimes F_\chi)
\]
be the coupling of $A_\tau$ to $F_\chi$.

Denote by $R(K)$ and $R(M)$ 
the representation rings of $K$ and $M$, respectively. Let $i\colon M\to K$
be the inclusion and $i^\ast\colon R(K)\to R(M)$ the induced map of the 
representation rings. The Weyl group $W(A)$ acts on $R(M)$ in the canonical 
way. Let $R^\pm(M)$ denote the $\pm1$-eigenspaces of the non-trivial element
$w\in W(A)$. Let $\sigma\in R(M)$. It follows from the proof of 
Proposition 1.1 in \cite{BO} that there exist $m_\tau(\sigma)\in\{-1,0,1\}$,
depending on $\tau\in \widehat K$, which are equal to zero except for finitely
many $\tau\in\widetilde K$, such that
\begin{equation}
\sigma=\sum_{\tau\in\widetilde K}m_\tau(\sigma) i^\ast(\tau), 
\end{equation}
if $\sigma\in R^+(M)$, and
\begin{equation}
\sigma+w\sigma=\sum_{\tau\in\widetilde K}m_\tau(\sigma) i^\ast(\tau), 
\end{equation}
if $\sigma\neq w\sigma$. 
Let
\begin{equation}\label{sigma-vb}
E(\sigma):=\bigoplus_{\substack{\tau\in\widetilde K\\ m_\tau(\sigma)=\neq0}} E_\tau.
\end{equation}
Then $E(\sigma)$ has a grading
\[
E(\sigma)=E^+(\sigma)\oplus E^-(\sigma)
\]
defined by the sign of $m_\tau(\sigma)$. 
Let $\sigma\in\widehat M$. Denote by $\nu_\sigma$ the highest weight of 
$\sigma$. Let ${\mathfrak b}$ be the standard Cartan subalgebra of $\mf$
\cite[Sect. 2]{MP1}. 
Let $\rho_{\mf}$ be the half-sum of positive roots of 
$(\mf_\C,{\mathfrak b}_\C)$. Put
\begin{equation}\label{c-sigma}
c(\sigma):=-\|\rho\|^2-\|\rho_{\mf}\|^2+\|\nu_\sigma+\rho_{\mf}\|^2.
\end{equation}
We define the operator $A^\sharp_{\chi}(\sigma)$ acting in 
$C^\infty(X,E(\sigma)\otimes F_\chi)$ by
\begin{equation}\label{oper1}
A^\sharp_{\chi}(\sigma):=\bigoplus_{\substack{\tau\in\widetilde K\\ m_\tau(\sigma)\neq0}}
A^\sharp_{\tau,\chi}+c(\sigma).
\end{equation}
For $p\in\{0,...,d-1\}$ let $\sigma_p$, be the standard representation of
$M=\SO(d-1)$ on $\Lambda^p\R^{d-1}\otimes\C$. Put
\begin{equation}\label{oper2}
A_\chi^\sharp(\sigma_p\otimes\sigma):=\bigoplus_{[\sigma^\prime]\in\widehat M/W}
\bigoplus_{i=1}^{[(\sigma_p\otimes\sigma)\colon\sigma^\prime]} 
A_\chi^\sharp(\sigma^\prime),
\end{equation}
Recall that $A_\chi^\sharp(\sigma_p\otimes\sigma)$ acts in the space of sections 
of a graded vector bundle.
Then by
\cite[Prop. 1.7]{Sp2} we have the following determinant formula.
\begin{prop}\label{prop-det}
For every $\sigma\in\widehat M$  one has
\begin{equation}\label{det-form1}
\begin{split}
R(s;\sigma,\chi)=\prod_{p=0}^{d-1} & \det_{\gr}\left(A_\chi^\sharp(\sigma_p\otimes\sigma)+
(s+n-p)^2\right)^{(-1)^p}\\
&\cdot\exp\left(-\frac{2\pi(n+1)\dim(V_\chi)\dim(V_\sigma)\vol(X)}{\vol(S^d)}s
\right),
\end{split}
\end{equation}
if $\sigma$ is Weyl-invariant, and
\begin{equation}\label{det-form2}
\begin{split}
R(s;\sigma,\chi)R(s;w\sigma,\chi)=\prod_{p=0}^{d-1} &\det_{\gr}\left(A_\chi^\sharp(\sigma_p\otimes\sigma)+(s+n-p)^2\right)^{(-1)^p}\\
&\cdot\exp\left(-\frac{4\pi(n+1)\dim(V_\chi)\dim(V_\sigma)\vol(X)}{\vol(S^d)}s\right),
\end{split}
\end{equation}
otherwise. Here $\vol(S^d)$ denotes the volume of the $d$-dimensional unit 
sphere.
\end{prop}
For unitary $\chi$ this was proved in \cite[Prop. 4.6]{BO}.

\section{Proof of the main theorem}\label{sec-origin}
\setcounter{equation}{0}

To prove Theorem \ref{theo-sing} we apply Proposition \ref{prop-det} for the 
case $\sigma=1$.
Let $\hf=\af\oplus\bfr$ be the standard Cartan subalgebra of $\gf$.  Let 
$e_1,...,e_{n+1}\in\hf^\ast_\C$ be the 
standard basis \cite[Sect. 2]{MP1}. Thus $e_2,...,e_{n+1}$ is a basis of 
$\bfr$. Then $\rho_{\mf}=\sum_{j=2}^{n+1}(n+1-j)e_j$ and 
\[
\nu_p=\begin{cases} e_2+\cdots+e_{p+1},& \text{if}\;\;p\le n,\\
e_2+\cdots+e_{2n+1-p},&\text{if}\;\; p>n,
\end{cases}
\]
\cite[Chap. IV, \S 7]{Kn}.
 Moreover, $|\rho|=n$. An explicit computation shows
that $c(\sigma_p)=-(n-p)^2$.  
Let $\lambda_p$ be the $p$-th exterior power of the standard representation
of $\SO(d)$ . Then for $p=0,...,d-1$ we have $i^\ast(\lambda_p)=\sigma_p+
\sigma_{p-1}$. Put $\tau_p:=\sum_{k=0}^p (-1)^k\lambda_{p-k}$. Then it follows that
$i^\ast(\tau_p)=\sigma_p$, $p=0,...,d-1$. Using \eqref{oper1} and \eqref{oper2},
we obtain
\begin{equation}\label{opera3}
A^\sharp_\chi(\sigma_p)+(n-p)^2=\bigoplus_{k=0}^p A^\sharp_{\lambda_k,\chi}.
\end{equation}
Now recall that $A^\sharp_{\lambda_k,\chi}$ is the coupling of $A_{\lambda_k}$ to 
$F_\chi$. Furthermore, with respect to the isomorphism  \eqref{iso-sect}, 
$A_{\lambda_k}$ corresponds to the action of $-R_\Gamma(\Omega)$ on
$(C^\infty(\Gamma\bs G)\otimes \Lambda^k\C^d)$. By the Lemma of Kuga, this
operator corresponds to the Laplacian $\Delta_k$ on $\Lambda^k(X)$. Let
$\Delta^\sharp_{k,\chi}$ be the coupling of $\Delta_k$ to $F_\chi$. Then by
\eqref{opera3} we get
\begin{equation}
A^\sharp_\chi(\sigma_p)+(n-p)^2=\bigoplus_{k=0}^p \Delta^\sharp_{k,\chi}.
\end{equation}
Using \eqref{det-form1} we obtain
\begin{equation}\label{det-form3}
\begin{split}
R(s;\chi)&=\prod_{p=0}^{d-1}\det_{\gr}(A^\sharp_\chi(\sigma_p)+(s+n-p)^2)^{(-1)^p}\\
&=\prod_{p=0}^{d-1}\prod_{k=0}^p\det(\Delta^\sharp_{p-k,\chi}+s(s+2(n-p)))^{(-1)^{p+k}}\\
&=\prod_{k=0}^{d-1}\prod_{p=k}^{d-1}\det(\Delta^\sharp_{k,\chi}+s(s+2(n-p)))^{(-1)^{k}}.
\end{split}
\end{equation}
Let $h_k$ be the dimension of the generalized eigenspace of 
$\Delta^\sharp_{k,\chi}$ with eigenvalue zero.
\begin{lem}\label{kernel}
We have $h_p=h_{d-p}$ for $p=0,...,d$.
\end{lem}
\begin{proof}
Let $\star\colon \Lambda^p(X,F_\chi)\to \Lambda^{d-p}(X,F_\chi)$ be the extension 
of the Hodge $\star$-star operator, which acts locally as 
$\star(\omega\otimes f)
=(\star\omega)\otimes f$, where $\omega$ is a usual $p$-form and $f$ a local
section of $F_\chi$. Since $\star\Delta_p=\Delta_{d-p}\star$, it follows from the
definition of the Laplacians coupled to $F_\chi$ that $\star\Delta^\sharp_{p,\chi}
=\Delta^\sharp_{d-p,\chi}\star$. It follows that for every $k\in\N$ we have
$\star(\Delta^\sharp_{p,\chi})^k=(\Delta^\sharp_{d-p,\chi})^k\star$.This proves the 
lemma.
\end{proof}
Denote by $h$ the order of the singularity of $R(s,\chi)$ at $s=0$. 
Using \eqref{det-form3} and Lemma \ref{kernel} it follows that 
\begin{equation}\label{order-sing}
h=\sum_{k=0}^{d-1}(d+1-k)(-1)^kh_k=\sum_{k=0}^n(d+1-2k)(-1)^kh_k.
\end{equation}

Let $\Pi_{k,0}$ be the spectral projection on the generalize eigenspace of 
$\Delta_{k,\chi}^\sharp$ with eigenvalue $0$. Let $(\Delta_{k,\chi}^\sharp)^\prime
:=(\Id-\Pi_{k,0})\Delta_{k,\chi}^\sharp$. We note that for $s\in\C$, $|s|\ll1$,
there is a common Agmon angle for the operator 
$(\Delta_{k,\chi}^\sharp)^\prime+s(s+2(n-p))$.
Therefore, in order to study the limit of 
$\det((\Delta_{k,\chi}^\sharp)^\prime+s(s+2(n-p)))$ as $s\to0$, we can use one
and the same Agmon angle.

 If $p\neq n$, we get
\begin{equation}\label{limit1}
\begin{split}
\lim_{s\to0} s^{-h_k}\det(\Delta^\sharp_{k,\chi}+s(s+2(n-p)))&=\lim_{s\to0}
\Bigl[\det((\Delta^\sharp_{k,\chi})^\prime+s(s+2(n-p))\\
&\hskip3.5truecm\cdot\frac{(s(s+2(n-p)))^{h_k}}{s^{h_k}}\Bigr]\\
&=(2(n-p))^{h_k}\cdot\det ((\Delta^\sharp_{k,\chi})^\prime).
\end{split}
\end{equation}
For $p=n$ we get a similar formula 
\begin{equation}\label{limt2}
\lim_{s\to0} s^{-2h_k}\det(\Delta^\sharp_{k,\chi}+s^2)=\lim_{s\to0}q
\det((\Delta^\sharp_{k,\chi})^\prime+s^2)
=\det ((\Delta^\sharp_{k,\chi})^\prime).
\end{equation}
Let
\begin{equation}\label{const}
C(d,\chi):=\prod_{k=0}^{d-1}\prod_{p=k}^{d-1} \left(2(n-p)\right)^{(-1)^kh_k}.
\end{equation}
Using \eqref{order-sing}, \eqref{limit1} and \eqref{limt2} we get
\begin{equation}\label{ruelle-zero}
\begin{split}
\lim_{s\to0} s^{-h}R(s;\chi)&=\prod_{k=0}^{d-1}
\prod_{\substack{p=k\\p\neq n }}^{d-1}
\lim_{s\to0}\left[s^{-h_k}\det(\Delta^\sharp_{k,\chi}+s(s+2(n-p)))\right]^{(-1)^{k}}\\
&\hskip3.5truecm\cdot \prod_{k=0}^n\lim_{s\to0}
\left[s^{-2h_k}\det(\Delta^\sharp_{k,\chi}+s^2)\right]^{(-1)^k}\\
&=C(d,\chi)\cdot\prod_{k=0}^{d-1}\det((\Delta_{k,\chi}^\sharp)^\prime)^{(d-k)(-1)^k}
=C(d,\chi)\cdot\prod_{k=1}^d\det((\Delta_{k,\chi}^\sharp)^\prime)^{k(-1)^{k+1}}.
\end{split}
\end{equation}
For the last equality we used that $\Delta_{k,\chi}^\sharp\cong
\Delta_{d-k,\chi}^\sharp$.
Let $T_0(X,\chi)$ be the torsion \eqref{double} of the double complex 
$(V_0^\ast,d,d^{\ast,\sharp}$ and
$T^\C(X,\chi)$ the Cappell-Miller torsion defined by \eqref{anal-tor}.
We note that $T^\C(X,\chi)$ and $T_0(X,\chi)$ are both non-zero elements of the
determinant line $\det H^\ast(X,F_\chi)\otimes(\det H_\ast(X,F_\chi))^\ast$. 
Hence there exists $\lambda\in\C$ with $T^\C(X,\chi)=\lambda T_0(X,\chi)$.
Set
\[
\frac{T^\C(X,\chi)}{T_0(X,\chi)}:=\lambda.
\]
If we combine this convention with the definition of the Cappell-Miller 
torsion \eqref{anal-tor1}, then \eqref{ruelle-zero} implies Theorem 
\ref{theo-sing}.

\section{Acyclic representations}\label{sec-acyclic}
\setcounter{equation}{0}

In this section we assume that $\chi$ is acyclic. Then
$T^\C(X,\chi)$, $T_0(X,\chi)$ and $\tau_{\comb}(X,\chi)$ are complex numbers and 
the right hand side of \eqref{origin} is the quotient of the two complex 
numbers. Besides the Cappell-Miller torsion we need another
version of a complex analytic torsion for arbitrary flat vector bundles 
$F_\chi$. This is the {\it refined analytic torsion} $T^{ran}(X,\chi)
\in\det(H^\ast(X,F_\chi)$ introduced by Braverman and Kappeler \cite{BK2}.  
The definition is based on the consideration of the odd signature operator 
$B_\chi$ \cite[2.1]{BK1}. It is defined as follows. Let
\[
\alpha\colon \Lambda^\ast(X,F_\chi)\to\Lambda^\ast(X,F_\chi)
\]
be the chirality operator defined by
\[
\alpha(\omega):=i^{n+1}(-1)^{k(k+1)/2}\star\omega,\quad \omega\in\Lambda^k(M,F_\chi).
\]
Let $\nabla_\chi$ be the flat connection in $F_\chi$. Then the odd signature
operator is defined as
\begin{equation}
B_\chi:=\alpha \nabla_\chi+\nabla_\chi\alpha.
\end{equation}
It leaves the even subspace $\Lambda^{\ev}(X,F_\chi)$ invariant. Let $B_{\ev,\chi}$
be the restriction of $B_\chi$ to $\Lambda^{\ev}(X,F_\chi)$. Then $T^{\ran}(X,\chi)
\in\det(H^\ast(X,F_\chi))$ is defined in terms of $B_{\ev,\chi}$. If $\chi$ is 
acyclic, then $ T^{\ran}(X,\chi)$ is a complex number.
In \cite{BK3}, Braverman and Kappeler determined the relation between the
Cappell-Miller torsion and the refined analytic torsion. 
Let $\eta(B)$ be the eta-invariant of $B_{\ev,\chi}$. In general,
$B_\chi$ is not self-adjoint and therefore, $\eta(B)$ is in general not real.
Furthermore, let $\eta_0$ be the eta-invariant of the trivial line bundle.
Then by Proposition 4.2 and Theorem 5.1 of \cite{BK3} it follows that
\begin{equation}\label{refined1}
T^\C(X,\chi)=\pm T^{\ran}(X,\chi)^2\cdot e^{-2\pi i(\eta(B)-\dim(\chi)\eta_0)}.
\end{equation}
On the other hand, it follows from \cite[Theorem 1.9]{BK2} that
\begin{equation}\label{refined2}
|T^{\ran}(X,\chi)|=T^{RS}(X,\chi)\cdot e^{\pi\Im(\eta(B))}.
\end{equation}
Combining \eqref{refined1} and \eqref{refined2}, we obtain \eqref{abs-value}.

\subsection{Restriction of repesentations of the underlying Lie group}
The first case that we consider are representations which are restictions to 
$\Gamma$ of representations of $G$.

Let $\rho\colon G\to \GL(V_\rho)$ be a 
finite dimensional real ( resp. complex) representation of $G$. Denote by
$F_\rho\to X$ the flat vector bundle associated to $\rho|_\Gamma$. Let
$\widetilde E_\rho\to G/K$ be the homogeneous vector bundle associated to 
$\rho|_K$. By \cite[Part I, Prop. 3.3]{MM} there is a canonical isomorphism
\begin{equation}\label{isom5}
F_\rho\cong \Gamma\bs \widetilde E_\rho.
\end{equation}
Let $\gf=\kf\oplus\pf$ be the Cartan decomposition of $\gf$. 
By \cite[Part I, Lemma 3.1]{MM}, there exists an inner product $\langle\cdot,
\cdot\rangle$ in $V_\rho$ such that
\begin{enumerate}
\item $\left<\tau(Y)u,v\right>=-\left<u,\tau(Y)v\right>$ for all 
$Y\in\mathfrak{k}$, $u,v\in V_{\tau}$
\item $\left<\tau(Y)u,v\right>=\left<u,\tau(Y)v\right>$ for all 
$Y\in\mathfrak{p}$, $u,v\in V_{\tau}$.
\end{enumerate}
Such an inner product is called admissible. It is unique up to scaling. Fix an 
admissible inner product. Since $\rho|_{K}$ is unitary with respect to this 
inner product, it induces a metric in $\Gamma\bs\widetilde E_{\tau}$ and
by \eqref{isom5} also in $F_\rho$. 
Denote by $T^{RS}(X,\rho)$ the Ray-Singer anaytic torsion
of $(X,F_\rho)$ with respect to metric on $X$ and the  metric in 
$F_\rho$. Denote by $\theta\colon G\to G$ the Cartan involution. Let 
$\rho_\theta:=\rho\circ\theta$. Assume that $\rho\not\cong\rho_\theta$. Then
$H^\ast(X,F_\rho)=0$, i.e., $\rho|_\Gamma$ is acyclic. In this case $T^{RS}(X,\rho)$
is independent of the metrics on $X$ and in $F_\rho$ \cite[Corollary 2.7]{Mu3}.
Let $R^\rho(s):=R(s,\rho)R(s,\rho_\theta)$. Then by \cite[Theorem 8.13]{Wo} 
$R^\rho(s)$ is holomorphic at $s=0$ and
\begin{equation}
R^\rho(0)=T^{RS}(X,\rho)^4.
\end{equation}
Furthermore, from the discussions in \cite[Sect. 9.1]{Wo} follows that 
both $R(s,\rho)$ and $R(s,\rho_\theta)$ are holomorphic at $s=0$ and 
$R(0,\rho_\theta)=\overline{R(0,\rho)}$. Thus it follows that
\begin{equation}\label{value-zero}
|R(0,\rho)|=T^{RS}(X,\rho)^2,
\end{equation}
which is \eqref{value-zero}.
Hence $R(s,\rho)$ is regular at $s=0$ and $R(0,\rho)\neq0$.
Applying Theorem \ref{theo-sing} we obtain Corollary \ref{cor-origin1}.

Next we briefly recall the definition of the Reidemeister torsion \cite{RS}.
We work with vector spaces over $\C$. Let $V$ be $\C$-vector space of dimension
$m$. Let $v=(v_1,...,v_m)$ and $w=(w_1,...,w_m)$ be two basis of $V$. Let 
$T=(t_{ij})$ be the matrix of the change of basis from $v$ to $W$, i.e.,
$w_i=\sum_j t_{ij}v_j$. Put $[W/v]:=\det(T)$. Let
\[
C^\ast\colon C^0\xrightarrow{\delta_0} C^1\xrightarrow{\delta_1}\cdots
\xrightarrow{\delta_{n-2}} C^{n-1}\xrightarrow{\delta_{n-1}}C^n
\]
be a cochain complex of finite dimensional complex vector spaces. 
Let $Z_q=\ker(\delta_q)$ and $B_q:=\Im(\delta_{q-1})\subset C^q$. 
Let $c_q$ (resp. $h_q$) be a preferred base of $C_q$ (resp. $H^q(C^\ast)$). 
Choose a basis $b_q$ for $B_q$, $q=0,...,n$, and let $\tilde b_{q+1}$ be an
independent set in $C_q$ such that $\delta_q(\tilde b_{q+1})=b_{q+1}$, and
let $\tilde h_q$ be an independent set in $Z_q$ which represents the base $h_q$
of $H^q(C^\ast)$. Then $(b_q,\tilde h_q,\tilde b_{q+1})$ is a basis of $C^q$ and
$[b_q,\tilde h_q,\tilde b_{q+1}/c_q]$  depends only on $b_q,h_q$ and $b_{q+1}$. 
Therefore, we denote it by $[b_q,h_q,b_{q+1}/c_q]$. Then the complex
Reidemeister torsion $\tau^\C(C_\ast)\in\C$ of the chain complex $C_\ast$ is  
defined by
\begin{equation}
\tau^\C(C^\ast):=\prod_{q=0}^n[b_q,h_q,b_{q+1}/c_q]^{(-1)^q}.
\end{equation}

Let $K$ be a $C^\infty$-triangulation of $X$ and $\widetilde K$ the lift of 
$K$ to a
triangulation of the univeral covering $\bH^d$ of $X$. Then 
$C^q(\widetilde K,\C)$ is a module over the complex group algebra $\C[\Gamma]$.
Now recall that $\rho$ is the restriction of a representation of $G$. Since
$G$ is a connected semisimple Lie group, it follows from \cite[Lemma 4.3]{Mu3}
that $\rho$ is a representation of $\Gamma$ in $\SL(N,\C)$.
Let
\[
C^q(K,\rho):=C^q(\widetilde K,\C)\otimes_{C[\Gamma]}\C^N
\]
the twisted cochain group and 
\[
C^\ast(K,\rho)\colon 0\to C^0(K,\rho)\xrightarrow{\partial_\rho}C^1(K,\rho)\xrightarrow{\partial_\rho}
\cdots\xrightarrow{\partial_\rho}C^d(K,\rho)\to 0
\]
the corresponding cochain complex. Let $e_1,...,e_{r_q}$ be a 
preferres basis of $C^q(\tilde K,\C)$ as a $\C[\Gamma]$-module consisting
of the duals of lifts of $q$-simpexes and let $v_1,...,v_N$ be a basis of
$\C_N$. Then $\{e_i\otimes v_j\colon i=1,...,r_q,\quad j=1,...,N\}$ is a 
preferres basis of $C^q(K,\rho)$. Now consider the complex-valued 
Reidemeister torsion $\tau^\C(C^\ast(K,\rho))$.
Since $\rho$ is a representation in $\SL(N,\C)$, a different choice of the 
preferred basis $\{e_i\}$ leads at most a to sign change of $\tau^\C(X,\rho)$.
If $v^\prime$ is a different basis of $\C^N$, then $\tau^\C(X,\rho)$ changes by
$[v^\prime/v]^{\chi(X)}$. Hence, if $\chi(X)=0$, $\tau^\C(C^\ast(K,\rho))$ is well
defined as an element of $\C^\ast/\{\pm1\}$. It depends only on the choice of 
the basis $h_q$ of $H^q(X,F_\rho))$. Since every two smooth triangulations of $X$ admit
a common subdivision, it follows from \cite{Mi} that $\tau^\C(C^\ast(K,\rho))$
is independent of the smooth triangulation $K$. Put
\begin{equation}
\tau^\C(X,\rho):=\tau^\C(C^\ast(K,\rho)).
\end{equation}
This is the complex-valued Reidemeister torsion of $X$ and $\rho$.
If $\rho\not\cong\rho_\theta$, then $H^\ast(X,F_\rho)=0$. Thus in this case
$\tau^\C(X,\rho)$ is a combinatorial invariant. It follows from property A,
satisfied by $\tau_{comb}(X,\rho)$ \cite[6.2]{CM}, that
\begin{equation}
\tau_{\comb}(X,\rho)=\tau^\C(X,\rho)^2.
\end{equation}
This implies \eqref{compl-reidem}.

\subsection{Deformations of acyclic unitary representations}
Let $\Rep(\Gamma,\C^n)$ be the set of all $n$-dimensional complex 
representations of $\Gamma$. It is well known that $\Rep(\Gamma,\C^n)$ has 
a natural structure of a complex algebraic variety \cite[13.6]{BK1}. Recall
that $\chi\in\Rep(\Gamma,\C^n)$ is called acyclic, if $H^\ast(X,F_\chi)=0$,
where $F_\chi\to X$ is the flat vector bundle associated to $\chi$. Denote
by $\Rep_0(\Gamma,\C^n)\subset \Rep(\Gamma,\C^n)$ the subset of all acyclic
representations. A representation $\chi\in\Rep(\Gamma,\C^n)$ is called 
unitary, if there exists a Hermitian scalar product $\langle\cdot,\cdot\rangle$
on $\C^n$ which is preserved by all maps $\chi(\gamma)$, $\gamma\in\Gamma$.
Let $\Rep^u_0(\Gamma,\C^n)\subset \Rep_0(\Gamma,\C^n)$ be the subset of all
unitary acyclic representations. By \cite[Theorem 1.1]{FN} we get
\begin{prop}\label{prop-acyclic1}
For every compact hyperbolic manifold $\Gamma\bs\bH^d$, we have
$\Rep_0^u(\Gamma,\C^n)\neq\emptyset$.
\end{prop}
Now let $\chi\in\Rep^u_0(\Gamma,\C^n)$.
For such a representation the flat Laplacian
$\Delta_{k,\chi}^\sharp$ equals the usual Laplace operator $\Delta_{k,\chi}$ and
$T^\C(X,\chi)=T^{RS}(X,\chi)^2$. Moreover, $h_k=0$, $k=0,...,d$, which implies
 $h=0$ and $T_0(X,\chi)=1$. Thus $R(s,\chi)$ is regular at $s=0$ and from 
Theorem \eqref{theo-sing} we recover Fried's result \cite{Fr1}
\begin{equation}
R(0,\chi)=T^{RS}(X,\chi)^2.
\end{equation}
We equip $\Rep(\Gamma,\C^n)$ with the topology obtained from its structure as
complex algebraic variety. The complement of the singular set is a complex
manifold.  Let $W\subset\Rep(\Gamma,\C^n)$ be the connected component of 
$\Rep(\Gamma,\C^n)$ which contains $\Rep^u_0(\Gamma,\C^n)$. Let $\chi_0\in W$
be a unitary acyclic representation and let $E_0$ be the associated flat vector
bundle. By \cite[Prop. 4.5]{GM} every vector bundles $E_\chi$, $\chi\in W$, is
isomorphic to $E_0$. Thus the flat connection  on $E_\chi$, which is
induced by the trivial connection on $\widetilde X\times\C^n$, corresponds
to a flat connection $\nabla_\chi$ on $E_0$. Now recall that 
\[
\Delta_\chi^\sharp=(d_\chi+d_\chi^{\ast,\sharp})^2,
\]
where
\[
d_\chi^{\ast,\sharp}\big|_{\Lambda^p(X,E_\chi)}=(-1)^p(\star\otimes\Id)d_\chi
(\star\otimes\Id).
\]
Via the isomorphism $E_\chi\cong E_0$, the operator $d_\chi+d_\chi^{\ast,\sharp}$ 
corresponds to the operator
\[
D_\chi^\sharp:=\nabla_\chi+\nabla_\chi^{\ast,\sharp}\colon \Lambda^\ast(X,E_0)\to\Lambda^\ast(X,E_0),
\]
where 
\[
\nabla_\chi^{\ast,\sharp}=(-1)^p(\star\otimes\Id)\nabla_\chi(\star\otimes\Id).
\]
Let $\nabla_0$ be the unitary flat connection on $E_0$. Let $\Co(E_0)$ denote
the space of connections on $E_0$. Recall that $\Co(E_0)$ can be identified 
with $\Lambda^1(X,\End(E_0))$ by associating to a connection $\nabla\in\Co(E_0)$
the $1$-form $\nabla-\nabla_0\in\Lambda^1(X,\End(E_0))$. We equipp $\Co(E_0)$
with the $C^0$-topology defined by the sup-norm $\|\omega\|_{\sup}:=\max_{x\in X}
|\omega(x)|$, $\omega\in\Lambda^1(X,\End(E_0))$, where $|\cdot|$ denotes the  
natural norm on $\Lambda^1(T^\ast X)\otimes E_0$. Since $E_0$ is acyclic, 
$D_0:=\nabla_0+\nabla_0^\ast$ is invertible. If $\|\nabla_\chi-\nabla_0\|\ll 1$
it follows as in  \cite[Prop 6.8]{BK1} that $D_\chi^\sharp$ is invertible and 
hence $\Delta_{\chi}^\sharp=(D_\chi^\sharp)^2$ is invertible too. 
Thus we get
\begin{lem}\label{lem-nbh}
There exists an open neighborhood $V\subset W$ of $\Rep_0^u(\Gamma,\C^n)$
such that $\Delta_{\chi}^\sharp$ is invertible for all $\chi\in V$.
\end{lem}
Let $\chi\in V$. Then we have $h_k=\dim(\Delta_{k,\chi}^\sharp)=0$, $k=0,...,d$,
and therefore the order $h$ of $R(s,\chi)$ at $s=0$ vanishes. Also $C=1$ and
$T_0(X,\chi)=1$. Thus by Theorem \ref{theo-sing} we obtain Proposition 
\ref{prop-acyclic}.

\end{document}